\documentclass[12pt]{amsart}
\usepackage{amsmath,amsfonts,amstext,amsthm,amssymb,pxfonts}
\usepackage[colorlinks,citecolor=red,hypertexnames=false]{hyperref}

\ifx\pdftexversion\undefined
\usepackage[dvips]{graphicx}
\else
\usepackage{graphicx}
\fi

\theoremstyle{plain} 

\newtheorem{theorem}[equation]{Theorem}

\theoremstyle{definition}
\newtheorem{definition}[equation]{Definition} 
\newtheorem{example}[equation]{Example}

\theoremstyle{remark}
\newtheorem{remark}[equation]{Remark}

\allowdisplaybreaks

\numberwithin{equation}{section}

\title{Chebyshev-Type Quadrature Formulas for New Weight Classes }

\author[ A.~Vagharshakyan ] {Armen Vagharshakyan }
\address{\newline Mathematics Department, \newline Brown University, \newline 151 Thayer St, \newline Providence, RI 02912 USA}
\email{armen@math.brown.edu}
\begin{document}

\begin{abstract}
We give Chebyshev-type quadrature formulas for certain new weight classes.  These formulas are of highest possible degree when the number of nodes is a power of 2. We also describe the  nodes in a constructive way, which is important for applications. One of our motivations to consider these type of problems is the Faraday cage phenomenon for discrete charges as discussed by J. Korevaar and his colleagues.
\end{abstract}
\maketitle
\textit{Keywords:} Chebyshev quadrature, Faraday cage phenomenon

\textit{AMS subject classification:} 65D32 (primary),  78A30 (secondary)

\section{\textbf{Introduction}}

	Let us denote by $Pol(m)$ the family of polynomials of degree $\leq m$.

\begin{definition} Let $\rho(x)$ be 
a nonnegative function on $[-1,1]$. Denote by:
 \begin{equation*}M_n(\rho)\end{equation*} 
the biggest natural number for which there exists a choice of distinct points $\lbrace x_k\rbrace_{k=1}^n\subset [-1,1]$ so that the equality:
\begin{equation}\label{e.chebyshevquadrature}
\int_{-1}^1 P(x)\rho(x)dx=\frac{1}{n}\sum_{k=1}^nP(x_k)
\end{equation}
is valid for an arbitrary polynomial $P\in Pol(M_n(\rho))$.
\end{definition}

\begin{remark}
	A quadrature formula of the form \eqref{e.chebyshevquadrature} is called Chebyshev-type quadrature formula, and the points $\lbrace x_k\rbrace_{k=1}^n\subset [-1,1]$ are called the nodes of the quadrature formula. One can consider quadrature formulas of a more general form:
\begin{equation}\label{e.gaussquadrature}
\int_{-1}^1 P(x)\rho(x)dx=\frac{1}{n}\sum_{k=1}^nA_kP(x_k)
\end{equation}
where $A_k$ are certain nonnegative numbers. These are known as Gauss type quadrature formulas. 
\end{remark}
	It is well known (see \cite{Krilov}, p. 97) that a Gauss type quadrature formula with $n$ nodes cannot be valid for an arbitrary $P\in Pol(2n)$. For completeness, we provide the proof of that fact here:

\begin{remark}\label{r.nottoobig}
Let $\rho(x)\geq 0$ be a nontrivial function on $ -1\leq x\leq 1$ . Then there are no points 
$x_k\in [-1,1],\,\, k=1,\dots,n$ and numbers $A_k,\,\, k=1,2,\dots,n$ such that the formula \eqref{e.gaussquadrature}
is valid for an arbitrary $P\in Pol(2n)$.
\end{remark}
\begin{proof}
Assume the converse, for some points 
$x_k\in [-1,1],\,\, k=1,2,\dots,n$ and numbers $A_k,\,\, k=1,2,\dots,n$
the formula \eqref{e.gaussquadrature} is true. Then for the polynomial:
\begin{equation*}
P_0(x)=\prod_{k=1}^{n}(x-x_k)^2
\end{equation*}
of degree $2n$ we must have:
\begin{equation*}
0<\int_{-1}^1 P_0(x)\rho(x)dx=\frac{1}{n}\sum_{k=1}^nA_kP(x_k)=0.
\end{equation*}
\end{proof}
\begin{remark}\label{r.nottoobig2}
	Remark \eqref{r.nottoobig} implies that $M_n(\rho)<2n$.
\end{remark}
	The first Chebyshev-type formula was first discovered by Mehler in 1864 (see \cite{Krilov}, p. 185):
\begin{theorem}[Mehler] The formula
\begin{equation*}
\frac{1}{\pi}\int_{-1}^1\frac{P(x)}{\sqrt{1-x^2}}dx=
\frac{1}{n}\sum_{k=1}^nP\left(\cos\frac{2k-1}{2n}\pi\right)
\end{equation*}
is valid for an arbitrary polynomial $P\in Pol(2n-1)$.
\end{theorem}
On the other hand in 1875 K. Posse proved that (see \cite{Posse} and \cite{Krilov}, p.186):

\begin{theorem}\label{t.posse}(Posse) If for a weight function \begin{equation*}\rho(x),\,\,\, x\in[-1,1],\end{equation*} we have that  $M_n(\rho)= 2n-1$ for all $n\in N$ then:
\begin{equation}\label{e.chebyshevweight}
\rho(x)=\frac{1}{\pi\sqrt{1-x^2}}.
\end{equation}
\end{theorem}

	The first weight function different from \eqref{e.chebyshevweight}
for which $M_n(\rho)\geq n$ was given 
by Ullman (see \cite{MR0205463}):
\begin{equation*}
\mu(x)=\frac{2}{\pi\sqrt{1-x^2}}\cdot\frac{1+bx}{1+b^2+2bx},\quad |b|<\frac{1}{2}
\end{equation*}
where
\begin{equation*}
M_n(\mu)=n.
\end{equation*}

	Later new examples of weights for whom $M_n(\rho)=n$ were found (see the survey article \cite{MR890266} and the references given there).
These examples of weight functions are mainly of the following form:
\begin{equation*}
\rho(x)=\frac{w(x)}{\sqrt{1-x^2}},
\end{equation*}
where $w(x)$ is a certain positive analytic function on $[-1,\,1]$.

	Examples for whom $w$ has a singularity at $0$ have been found, too (see \cite{MR1379132}).

	The main result of this article - theorem \eqref{t.main} provides Chebyshev-type quadrature formulas for certain new weight classes 
	${\bf W}_n$. 
	
	All these classes include not only the weight:
	\begin{equation}\label{e.chebyshev}
\rho(x)=\frac{1}{\pi\sqrt{1-x^2}}
\end{equation}
	(see remark \eqref{r.properinclusion}), but also many others (see remark \eqref{r.largeclass}).
		
	 Our result compares to the above-mentioned results by \eqref{r.comparison}. 
	 
	 The Chebyshev-type quadrature formulas for ${\bf W}_n$ are of highest possible degree when the number of nodes is a power of 2 (see remark \eqref{r.main2}). Thus, even though \eqref{e.chebyshev} is the only weight for which the Chabyshev type quadrature formula is of the highest possible degree $M_n(\rho)= 2n-1$ (see theorem \eqref{t.posse}), yet we prove that for every $n$ which is a power of 2 there are many more weights for whom there exists a Chebyshev-type quadrature formula of the highest possible degree. \newline
	
	Also note that the nodes in  theorem \eqref{t.main} are described in a constructive way, which is important for applications. 
\newline
\section{\textbf{The Case $\rho\equiv 1/2$} }
Let us discuss the case $\rho\equiv 1/2$ separately. S. Bernstein (see \cite{Krilov}, page 197) proved that if $n=8$ or $n\geq 10$ then one cannot choose points $x_{k,n},\,\, k=1,2,\dots, n$ in the segment $[-1,1]$ so that the following formula:
\begin{equation*}
\frac{1}{2}\int_{-1}^1P(x)dx=\frac{1}{n}\sum_{k=1}^nP(x_{k,n})
\end{equation*}
is valid for all polynomials of degree $n$.

	S. Bernstein (see \cite{Bernstein2}) proved the following inequality  $M_n(1/2)$:
\begin{equation*}
M_n(1/2)<\pi \sqrt{2n}.
\end{equation*}

	A. Kuijlaars \cite{MR1240246} (using the methods of S. Bernstein \cite{Bernstein1} ) proved that:

\begin{theorem}\label{t.bernstein} 
There exists an absolute constant $c>0$ such that:
\begin{equation*}
M_n(1/2)> c\sqrt{n}.
\end{equation*}
\end{theorem}

	Today the interest into Chebyshev-type quadrature formulas 
is explained not only by numerous nontrivial conjectures relating to the issue (see \cite{MR1379132},\cite{MR1421438} for a list of open problems) but also
	a connection between the Faraday cage phenomenon for discrete charges and  
Chebyshev-type quadrature formulas, as explained by J. Korevaar and his colleagues (see \cite{MR1473445}). 
For a more detailed explanation we refer the reader to the article \cite{MR1473445}.

This problem is also related to the discrepancy theory (see \cite{MR1439152}).
\newline
\section{\textbf{Auxiliary Constructions}}
	Let us denote by $x_0=-1$ and
\begin{equation*}
x_{n+1}=\sqrt{\frac{1+x_n}{2}},\,\,\, n=0,1,2,\dots
\end{equation*}
This sequence is increasing and
\begin{equation*}
\lim_{n\to \infty}x_n=1.
\end{equation*}
\begin{definition}\label{d.polynomial}
 For an arbitrary natural $n$ let us denote by $P_n(x)$ the 
polynomial of degree $n$, defined in the following way:
\newline
1. we put:
\begin{equation*}
P_0(x)=1,\quad P_{2}(x)=2x^2-1,
\end{equation*}
2. for an odd number $q$ let:
\begin{equation*}
P_q(x)=x^q,
\end{equation*}
3. for a natural number $n=2^p$ let:
\begin{equation*}
P_{2^p}(x)=P_{2}\left(P_{2^{p-1}}(x)\right),
\end{equation*}
4. For an arbitrary natural number $n=2^p q$, where $q=1(mod 2)$ and $3\leq q$ let:
\begin{equation*}
P_n(x)=\left(P_{2^p}(x)\right)^q.
\end{equation*}
For $n=2^p q$, where $q=1(mod 2)$ denote:
\begin{equation*}
|P_n|=p.
\end{equation*}
\end{definition}
\begin{example}
 We have:
\begin{equation*}
P_0(x)=1,\quad P_1(x)=x,\quad P_2(x)=2x^2-1,\quad P_3(x)=x^3,
\end{equation*}\begin{equation*}
P_4(x)=2\left(2x^2-1\right)^2-1,\quad P_5(x)=x^5,\quad P_6(x)=\left(2x^2-1\right)^3,
\end{equation*}
\begin{equation*}
P_7(x)=x^7,\quad P_8(x)=2\left(2\left(2x^2-1\right)^2-1\right)^2-1
\end{equation*} 
\end{example}
\begin{theorem} For an arbitrary natural $n$ we have:
\begin{equation*}
\max_{0\leq k\leq n}|P_k|=[\log_2 n].
\end{equation*}
\end{theorem}
\begin{proof}
 We present each number $1\leq k\leq n$ n the form 
$k=2^{p_k} q_{k}$, where $q_k=1(mod 2)$. Then:
\begin{equation*}
\max_{0\leq k\leq n}|P_k|=\max_{0\leq k\leq n}|p_k|=[\log_2 n].
\end{equation*}
\end{proof}

\begin{example}
 We have:
\begin{equation*}
|P_{2^k}|=k,\quad |P_{2^k-1}|=k-1.
\end{equation*}
\end{example}
\begin{theorem} For an arbitrary natural $1\leq p$ we have:
\begin{equation*}
P_{2^p}(x_{p})=-1,\quad P_{2^p}(x_{p+1})=0,\quad P_{2^p}(1)=1.
\end{equation*}
\end{theorem}
On the interval $[x_{p},\,\,1]$ the 
polynomial $P_{2^p}(x)$ is an increasing function.
\begin{definition} Let
\begin{equation*}
S_{0}:\left[x_{1},1\right]\rightarrow \left[x_{0},x_{1}\right],
\end{equation*}
and for each $x\in\left[x_{1},1\right]$,
\begin{equation*}
S_{0}(x)=-x.
\end{equation*}
For a natural $n$ let us denote by:
\begin{equation*}
S_{n}:\left[x_{n+1},1\right]\rightarrow \left[x_{n},x_{n+1}\right].
\end{equation*}
the one-to-one mapping, such that:
\begin{equation*}
P_{2^n}\left(x\right)=-P_{2^n}\left(S_{n}(x)\right),\,\,\,x_{n+1}\leq x\leq 1.
\end{equation*}
\end{definition}

	We have:
\begin{equation*}
S_{n}(x_{n+1})=x_{n+1},\quad S_{n}(1)=x_{n},\,\,\, n=1,2,\dots
\end{equation*}
\begin{figure}[ht]
\centering 
\includegraphics[height=90mm]{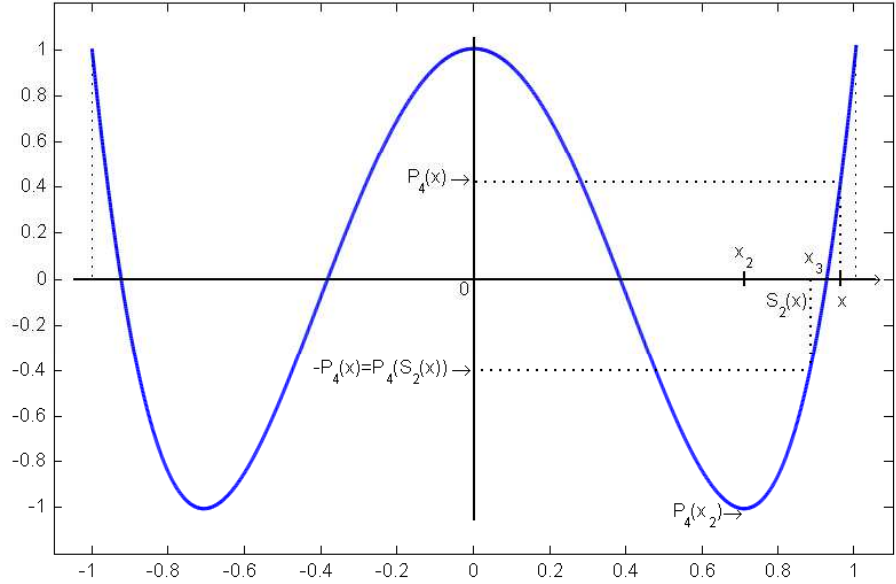}
\caption{The mapping $S_2$}
\end{figure}
\begin{example}
We have:
\begin{equation*}
S_1(y)=-y,\quad \frac{dS_1(y)}{dy}=-1,\quad x_2<y<1,
\end{equation*}
\begin{equation*}
S_2(y)=\sqrt{1-y^2},\quad \frac{dS_2(y)}{dy}=-\frac{y}{\sqrt{1-y^2}},\quad x_3<y<1,
\end{equation*}
\begin{equation*}
S_3(y)=\frac{y+\sqrt{1-y^2}}{\sqrt{2}},\quad \frac{dS_3(y)}{dy}=-\frac{y-\sqrt{1-y^2}}{\sqrt{2(1-y^2)}},\quad x_4<y<1.
\end{equation*}
\end{example}
\begin{definition}\label{d.definer} For an arbitrary integral number $0\leq n$ and a function $f(x)$ defined on the interval
$x_{n}\leq x\leq 1,$ let us denote by $R_n(f)$ the function defined on the interval $x_{n+1}<y<1$ by:
\begin{equation*}
R_n(f)(y)=\frac{f\left(S_n(y)\right)+f(y)}{2}.
\end{equation*}
\end{definition}
\begin{remark}
If the function $f$ is constant on the interval $[-1,1]$ then for an arbitrary $0\leq n$ we have:
\begin{equation*}
R_n(f)(y)=f(y),\quad \text{on}\quad x_{n+1}<y<1.
\end{equation*}
\end{remark}
\begin{remark}
For any of the polynomials $P_k(x),\,\,\,k=0,1,2,\dots$ which we constructed,
we have:
\begin{equation*}
R_0(P_k)(x)=P_k(x),\quad k=0(mod 2),
\end{equation*}
and
\begin{equation*}
R_0(P_k)(x)=0,\quad k=1(mod 2).
\end{equation*}
\end{remark}
\begin{remark}\label{r.ractsonq}
For any of the polynomials $P_{2^pk}(x),\,\,\,k,p=1,2,\dots$, which we constructed,
we have:
\begin{equation*}
R_p(P_{2^pk})(x)=P_{2^pk}(x),\quad k=0(mod 2),
\end{equation*}
and
\begin{equation*}
R_p(P_{2^pk})(x)=0,\quad k=1(mod 2).
\end{equation*}
\end{remark}
\begin{definition} For natural $n$ we'll denote by ${\bf W}_n$ the family of nonnegative functions  $\rho(x),\,\,\,-1<x<1$  satisfying the following conditions:
\begin{equation*}
\int_{-1}^1\rho(x)dx=1
\end{equation*}
and
\begin{equation*}
\rho(y)=-\rho(S_k(y))\frac{dS_k(y)}{dy},\,\,x_{k}<y\leq 1,
\end{equation*}
for $k=0,1,2,\dots,n-1$. 
\end{definition}
\begin{remark}\label{r.largeclass}
	Let us note that for an arbitrary weight function $\rho(x)\in {\bf W}_n$ we have:
\begin{equation}
\int_{x_n}^1\rho(y)dy=2^{-n}.
\end{equation}
Moreover, an arbitrary nonnegative function defined on $[x_n,\,1]$ satisfying the condition (3.14), 
coincides with a weight function from ${\bf W}_n$ on $[x_n,\,1]$.
\end{remark}
\begin{example}\label{example.even}
Note, that:
\begin{equation*}
S_0(y)=-y,\quad x_1\leq y\leq 1,
\end{equation*}
therefore for the weight function $\rho\in {\bf W}_1$ we get the condition:
\begin{equation*}
\rho(y)=\rho(-y),\,\,x_{0}<y\leq 1.
\end{equation*}
\end{example}
\begin{example}\label{example1}
We have:
\begin{equation*}
S_1(y)=\sqrt{1-y^2},\quad x_2\leq y\leq 1,
\end{equation*}
therefore for the weight function $\rho\in {\bf W}_2$ we get the conditions:
\begin{equation*}
\rho(y)=\rho(-y),\,\,x_{0}=-1<y\leq 1,
\end{equation*}
and
\begin{equation*}
\rho(y)=\rho\left(\sqrt{1-y^2}\right)\frac{y}{\sqrt{1-y^2}},\,\,x_{1}<y\leq 1,
\end{equation*}
	In particular, the function:
\begin{equation*}
\rho(y)=\frac{1}{\pi\sqrt{1-y^2}}
\end{equation*}
satisfies these conditions.
\end{example}
\begin{example}\label{example2}
We have:
\begin{equation*}
S_2(y)=\frac{y+\sqrt{1-y^2}}{\sqrt{2}},\quad x_3\leq y\leq 1,
\end{equation*}
therefore for the weight function $\rho\in {\bf W}_3$ we get the conditions:
\begin{equation*}
\rho(y)=\rho(-y),\,\,x_{0}<y\leq 1,
\end{equation*}
and
\begin{equation*}
\rho(y)=\rho\left(\sqrt{1-y^2}\right)\frac{y}{\sqrt{1-y^2}},\,\,x_{1}<y\leq 1,
\end{equation*}
and
\begin{equation*}
\rho(y)=\rho\left(\frac{y+\sqrt{1-y^2}}{\sqrt{2}}\right)\frac{y-\sqrt{1-y^2}}{\sqrt{2(1-y^2)}},
\quad x_2<y<1.
\end{equation*}

	In particular, the following function:
\begin{equation*}
\rho(y)=\frac{1}{\pi\sqrt{1-y^2}}
\end{equation*}
satisfies all the three conditions.
\end{example}
\begin{remark}\label{r.properinclusion}
By a reasoning similar to \eqref{example1} and \eqref{example2} one can prove that:
\begin{equation*}
\rho(y)=\frac{1}{\pi\sqrt{1-y^2}}\in {\bf W}_n \;\text{for}\; n=1,2,\dots.
\end{equation*}
\end{remark}
\begin{remark}\label{r.comparison} The class ${\bf W}_n$ coincides with the family of non-negative 
functions $\rho(x)$ whose integral is equal to $1$ and who permit the following representation: 
\begin{equation*}
\rho(x)=\frac{w(x)}{\sqrt{1-x^2}},
\end{equation*}
where $w(x)$ satisfies the conditions:
\begin{equation*}
w(x)=w(S_k(x)),\quad\text{for all}~ k=0,1,\dots,n-1, ~\text{and}~  x_{k+1}<x\leq 1.
\end{equation*}
\end{remark}
\section{\textbf{Main Results}}
\begin{theorem}\label{t.equality}
Let $n\geq 1$ be a natural number. Let $\rho(x)\in {\bf W}_n$. 
Then for an arbitrary function $f(x)$ we have:
\begin{equation*}
\int_{-1}^1f(x)\rho(x)dx=2^n\int_{x_n}^1 (R_{n-1}\circ\dots\circ R_0)(f)(x)\cdot\rho(x)dx.
\end{equation*}
\end{theorem}
\begin{proof} If $\rho(x)\in {\bf W}_1$ we have $\rho(y)=\rho(-y)$ so,
\begin{equation*}
\int_{-1}^1f(x)\rho(x)dx=\int_{-1}^0f(x)\rho(x)dx+\int_{0}^1f(x)\rho(x)dx=
\end{equation*}
\begin{equation*}
=\int_{0}^{1}f(x)\rho(x)dx-\int_{0}^{1}f\left(S_0(y)\right)\rho\left(S_0(y)\right)\frac{dS_0(y)}{dy}dy=
\end{equation*}
\begin{equation*}
=\int_{0}^{1}(f(x)+f(-x))\rho(x)dx=2\int_{x_1}^{1}R_0(f)(y)\rho(y)dy.
\end{equation*}

	If $\rho(x)\in {\bf W}_2$ then we have $\rho(y)=\rho(-y)$ and
\begin{equation*}
\rho(y)=-\rho\left(S_1(y)\right)\frac{dS_1(y)}{dy},\,\,x_2\leq y\leq 1.
\end{equation*}
Thus, for an arbitrary function $f(x)$ we have:
\begin{equation*}
\int_{-1}^1f(x)\rho(x)dx=2\int_{x_1}^1 R_0(f)(x)\rho(x)dx=2\int_{x_1}^{x_2}R_0(f)(x)\rho(x)dx+2\int_{x_2}^{1}R_0(f)(x)\rho(x)dx=
\end{equation*}
\begin{equation*}
=2\int_{x_2}^{1}R_0(f)(x)\rho(x)dx-2\int_{x_2}^{1}R_0(f)\left(S_1(y)\right)\rho\left(S_1(y)\right)\frac{dS_1(y)}{dy}dy=
\end{equation*}
\begin{equation*}
=2\int_{x_2}^{1}\left(R_0(f)\left(S_1(y)\right)+R_0(f)(y)\right)\rho(y)dy=2^2\int_{x_2}^{1}(R_1\circ R_0)(f)(y)\rho(y)dy.
\end{equation*}

	We prove the theorem for the other values of $n$ in an analogous way.
\end{proof}
\begin{theorem}\label{t.main} Let $\rho(x)\in {\bf W}_n$. Let 
$X=\lbrace{t_j \rbrace}_{j=1}^m$ be a certain set of points, for whom:
\begin{equation*}
-1<t_1<t_2<\dots<t_m<1.
\end{equation*}
Let $X\cap \left[x_{n},1\right]\neq \emptyset $ and
\begin{equation*}
S_k \left(X\cap \left[x_{k+1},1\right]\right)= X\cap \left[x_{k},x_{k+1}\right],\quad \text{for all} \quad k=0,1,2,\dots,n-1.
\end{equation*}
Then for an arbitrary polynomial $P(x)\in Pol(2^n-1)$ we have:
\begin{equation*}
\int_{-1}^1P(x)\rho(x)dx=\frac{1}{m}\sum_{k=1}^mP(t_k).
\end{equation*}
\end{theorem}

\begin{proof}

Remark \eqref{r.ractsonq} implies that:
\begin{equation*}
\left(R_{n-1}\circ \dots \circ R_0\right)P_k= 0, \quad \text{for} \quad k=1,\dots, 2^n-1.
\end{equation*}
So, for each polynomial $P(x)\in Pol\left(2^n-1\right)$ we have
$\left(R_{n-1}\circ \dots \circ R_0\right)P= const $.

Thus, by theorem \eqref{t.equality} we have:
\begin{equation*}
\int_{-1}^1P(x)\rho(x)dx=(R_{n-1}\circ\dots\circ R_0)(P)(t)\cdot 2^n\int_{x_n}^1 \rho(x)dx,
\end{equation*}
for any $x_n\leq t<1$. Further, by \eqref{r.largeclass} we get:
\begin{equation*}
\int_{-1}^1P(x)\rho(x)dx=(R_{n-1}\circ\dots\circ R_0)(P)(t),
\end{equation*}
for any $x_n\leq t<1$. This is why,
\begin{equation*}
\int_{-1}^1P(x)\rho(x)dx=\frac{1}{|X\cap \left[x_{n},1\right]|}\sum_{x_n\leq t_j< 1}(R_{n-1}\circ\dots\circ R_0)(P)(t_j).
\end{equation*}
At this moment let's consider the particular case $n=1$. From \eqref{d.definer} we have:
\begin{equation*}
\frac{1}{|X\cap \left[0,1\right]|}\sum_{0\leq t_j< 1}R_0(P)(t_j)=\frac{1}{2|X\cap \left[0,1\right]|}\sum_{0\leq t_j< 1}\left(P(t_j)+P(-t_j)\right)=\frac{1}{m}\sum_{k=1}^mP(t_k).
\end{equation*}
 For the other values of $n$ the equality:
\begin{equation*}
\frac{1}{|X\cap \left[x_{n},1\right]|}\sum_{x_n\leq t_j< 1}(R_{n-1}\circ\dots\circ R_0)(P)(t_j)=\frac{1}{m}\sum_{k=1}^mP(t_k).
\end{equation*}
is proved in an analogous way.

	\end{proof}
\begin{remark}\label{r.main}
Theorem  \eqref{t.main}  implies that for any $\rho\in \bf W_n$ we have:
\begin{equation*}
m=|X|\geq 2^{n-1}.
\end{equation*}
If we additionally impose that:
\begin{equation*}
t_m=x_n,
\end{equation*}
then we get that $|X|=m=2^{n-1}$ and the quadrature formula 
is valid for an arbitrary poynomial $P(x)\in Pol\left(2^n-1\right)$. Hence, 
\begin{equation*}
M_{2^{n-1}}(\rho)\geq 2^n-1
\end{equation*}

or 
\begin{equation*}
M_m(\rho)\geq 2m-1,\quad \text{for}\quad m=2^{n-1}.
\end{equation*}

	But taking into account the remark \eqref{r.nottoobig2}, we get that:
	 \begin{equation*}
M_m(\rho)= 2m-1,\quad \text{for}\quad m=2^{n-1}.
\end{equation*}
\end{remark}
\begin{remark}\label{r.main2}
Remark \eqref{r.main} shows that we found a Chebyshev - type quadrature formula for  $\rho\in \bf W_n$  of the highest possible degree $2^n-1$, where the number of nodes is $m=2^{n-1}$. 
\end{remark}



\begin{thebibliography}{99}





\bibitem{MR1421438}
    {\sc J. Korevaar and J.L.H. Meyers},
   {\em Spherical {F}araday cage for the case of equal point charges
              and {C}hebyshev-type quadrature on the sphere},
Integral Transform. Spec. Funct.
    \textbf{1}
     (1993),
  no. 2,
105--117.




\bibitem{MR1473445}
    {\sc J. Korevaar and M.A. Monterie},
     {\em Approximation of the equilibrium distribution by distributions
              of equal point charges with minimal energy},
Trans. Amer. Math. Soc.
    \textbf{350}
      (1998),
    no. 6,
 2329--2348.



\bibitem{MR1240246}
{\sc A. Kuijlaars},
 {\em The minimal number of nodes in {C}hebyshev type quadrature
              formulas},
Indag. Math. (N.S.)
\textbf {4}
(1993),
no. 3,
339--362.




\bibitem{MR1379132}
  {\sc A. B. J. Kuijlaars},
  {\em Chebyshev quadrature for measures with a strong singularity},
Proceedings of the {I}nternational {C}onference on
              {O}rthogonality, {M}oment {P}roblems and {C}ontinued
              {F}ractions ({D}elft, 1994),
J. Comput. Appl. Math.
\textbf {65}
(1995),
no. 1-3,
207-214.


\bibitem{MR0205463}
   {\sc J.L. Ullman},
 {\em A class of weight functions that admit {T}chebycheff
              quadrature},
Michigan Math. J.
\textbf{13}
(1966),
417--423.
    


\bibitem{MR890266}
 {\sc K.-J. F{\"o}rster},
 {\em Variace in quadrature - a survey},
Numerical Integration
 \textbf{4} (Birkhauser, Basel)
(1993),
    pp 91-110.



\bibitem{MR1439152}
{\sc E. B. Saff and A. B. J. Kuijlaars},
{\em Distributing many points on a sphere},
Math. Intelligencer
 \textbf{19}
(1997),
    no. 1,
5-11.


\bibitem{Bernstein1}
{\sc S.N. Bernstein},
{\em On quadrature formulas with positive coefficients (in Russian)},
Izv. Akad. Nauk SSSR Ser. Mat.
\textbf{4}
(1937), 
479-503.


\bibitem{Bernstein2}
{\sc S.N. Bernstein},
{\em Sur la formule de quadrature approchŽe de Tchebycheff },
Comptes Rendus
\textbf{203}
(1936), 1305-1306.



%
\bibitem{Krilov}
{\sc V. I. Krilov},
{\em Priblizhennoe vychislenie integralov (in Russian)},
Moscow,
(1959).

\bibitem{Posse}
{\sc K.A. Posse},
{\em Sur les quadratures},
Nouv. Ann. de Math.
\textbf{26}
(1875), 49-62.




\end{thebibliography}
\end{document}